\documentclass[a4paper]{amsart}
\usepackage{amscd}
\usepackage{amsmath}
\usepackage{amssymb}
\usepackage{amsthm}
\usepackage{bbm}
\usepackage{stmaryrd}

\usepackage[T1]{fontenc}

\newif\ifpdf
\ifx\pdfoutput\undefined
   \pdffalse        % we are not running PDFLaTeX
\else
   \pdfoutput=1     % we are running PDFLaTeX
   \pdftrue
\fi

\ifpdf
   \usepackage[pdftex]{graphicx}
\else
   \usepackage{graphicx}
\fi

\numberwithin{equation}{section} \swapnumbers

\newtheorem{satz}{Satz}[section]

\newtheorem{theorem}[satz]{Theorem}
\newtheorem{proposition}[satz]{Proposition}
\newtheorem{corollary}[satz]{Corollary}
\newtheorem{lemma}[satz]{Lemma}
\newtheorem{assumption}[satz]{Assumption}

\newtheorem{definition}[satz]{Definition}

\newtheorem{remark}[satz]{Remark}

\begin{document}

\title[The Yamada-Watanabe Theorem for SPDEs]{The Yamada-Watanabe Theorem for mild solutions to stochastic partial differential equations}
\author{Stefan Tappe}
\address{Leibniz Universit\"{a}t Hannover, Institut f\"{u}r Mathematische Stochastik, Welfengarten 1, 30167 Hannover, Germany}
\email{tappe@stochastik.uni-hannover.de}
\thanks{The author is grateful to an anonymous referee for valuable comments.}
\begin{abstract}
We prove the Yamada-Watanabe Theorem for semilinear stochastic partial differential equations with path-dependent coefficients. The so-called ``method of the moving frame'' allows us to reduce the proof to the Yamada-Watanabe Theorem for stochastic differential equations in infinite dimensions.
\end{abstract}
\maketitle

\section{Introduction}\label{sec-intro}

The goal of the present paper is to establish the Yamada-Watanabe Theorem -- which originates from the paper \cite{Yamada} -- for mild solutions to semilinear stochastic partial differential equations (SPDEs)
\begin{align}\label{SPDE}
dX(t) = (A X(t) + \alpha(t,X)) dt + \sigma(t,X)dW(t)
\end{align}
in the spirit of \cite{Da_Prato, Prevot-Roeckner, Atma-book} with path-dependent coefficients.
More precisely, denoting by $H$ the state space of (\ref{SPDE}), we will prove the following result (see, e.g. \cite{Ikeda} for the finite dimensional case):

\begin{theorem}\label{thm-main}
The SPDE (\ref{SPDE}) has a unique mild solution if and only if both of the following two conditions are satisfied:
\begin{enumerate}
\item For each probability measure $\mu$ on $(H,\mathcal{B}(H))$ there exists a martingale solution $(X,W)$ to (\ref{SPDE}) such that $\mu$ is the distribution of $X(0)$.

\item Pathwise uniqueness for (\ref{SPDE}) holds.
\end{enumerate}
\end{theorem}

The precise conditions on $A$, $\alpha$ and $\sigma$, under which Theorem \ref{thm-main} holds true, are stated in Assumptions~\ref{ass-1} and \ref{ass-2} below.
So far, the following two versions of the Yamada-Watanabe Theorem in infinite dimensions are known in the literature:
\begin{itemize}
\item For SPDEs of the type (\ref{SPDE}) with state-dependent coefficients $\alpha(t,X(t))$ and $\sigma(t,X(t))$; see \cite{Ondrejat}.

\item For stochastic evolution equations in the framework of the variational approach; see \cite{Roeckner}.
\end{itemize}
We will divide the proof of Theorem~\ref{thm-main} into two steps:
\begin{enumerate}
\item First, we show that we can reduce the proof to Hilbert space valued SDEs
\begin{align}\label{SDE}
dY_t = \bar{\alpha}(t,Y) dt + \bar{\sigma}(t,Y) dW_t.
\end{align}
This is due to the ``method of the moving frame'', which has been presented in \cite{SPDE}, see also \cite{Tappe-Refine}.

\item For Hilbert space valued SDEs (\ref{SDE}) however, the Yamada-Watanabe Theorem is a consequence of \cite{Roeckner}.
\end{enumerate}
The remainder of this paper is organized as follows: In Section~\ref{sec-framework} we present the general framework, in Section~\ref{sec-proof} we provide the proof of Theorem~\ref{thm-main}, and in Section~\ref{sec-example} we show an example illustrating Theorem~\ref{thm-main}.

\section{Framework and definitions}\label{sec-framework}

In this section, we prepare the required framework and definitions. The framework is similar to that in \cite{Roeckner} and we refer to this paper for further details.

Let $H$ be a separable Hilbert space and let $(S_t)_{t \geq 0}$ be a $C_0$-semigroup on $H$ with infinitesimal generator $A : \mathcal{D}(A) \subset H \rightarrow H$. The path space
\begin{align*}
\mathbb{W}(H) := C(\mathbb{R}_+;H)
\end{align*}
is the space of all continuous functions from $\mathbb{R}_+$ to $H$. Equipped with the metric
\begin{align}\label{metric}
\rho(w_1,w_2) := \sum_{k=1}^{\infty} 2^{-k} \Big( \sup_{t \in [0,k]} \| w_1(t) - w_2(t) \| \wedge 1 \Big),
\end{align}
the path space $(\mathbb{W}(H),\rho)$ is a Polish space. Furthermore, we define the subspace
\begin{align*}
\mathbb{W}_0(H) := \{ w \in \mathbb{W}(H) : w(0) = 0 \}
\end{align*}
consisting of all functions from the path space $\mathbb{W}(H)$ starting in zero. For $t \in \mathbb{R}_+$ we denote by
$\mathcal{B}_t(\mathbb{W}(H))$ the $\sigma$-algebra generated by all maps $\mathbb{W}(H) \rightarrow H$, $w \mapsto w(s)$ for $s \in [0,t]$.
Let $\mathcal{C}(H)$ be the collection of all cylinder sets of the form
\begin{align}\label{cyl-1}
\{ w \in \mathbb{W}(H) : w(t_1) \in B_1, \ldots, w(t_n) \in B_n \}
\end{align}
with $t_1,\ldots,t_n \in \mathbb{R}_+$ and $B_1,\ldots,B_n \in \mathcal{B}(H)$ for some $n \in \mathbb{N}$, and let $\mathcal{C}'(H)$ be the collection of all cylinder sets of the form
\begin{align}\label{cyl-2}
\{ w \in \mathbb{W}(H) : (w(t_1), \ldots, w(t_n)) \in B \}
\end{align}
for $t_1,\ldots,t_n \in \mathbb{R}_+$ and $B \in \mathcal{B}(H)^{\otimes n}$ for some $n \in \mathbb{N}$.
Similarly, for $t \in \mathbb{R}_+$ let $\mathcal{C}_t(H)$ be the collection of all cylinder sets of the form (\ref{cyl-1}) with $t_1,\ldots,t_n \in [0,t]$ and $B_1,\ldots,B_n \in \mathcal{B}(H)$ for some $n \in \mathbb{N}$, and let $\mathcal{C}_t'(H)$ be the collection of all cylinder sets of the form (\ref{cyl-2}) for $t_1,\ldots,t_n \in [0,t]$ and $B \in \mathcal{B}(H)^{\otimes n}$ for some $n \in \mathbb{N}$.

\begin{lemma}\label{lemma-cylinders}
The following statements are true:
\begin{enumerate}
\item We have $\mathcal{B}(\mathbb{W}(H)) = \sigma(\mathcal{C}(H)) = \sigma(\mathcal{C}'(H))$.

\item We have $\mathcal{B}_t(\mathbb{W}(H)) = \sigma(\mathcal{C}_t(H)) = \sigma(\mathcal{C}_t'(H))$ for each $t \in \mathbb{R}_+$.
\end{enumerate}
\end{lemma}

\begin{proof}
We can argue as in the finite dimensional case, see e.g. \cite[Section 2.II]{Shiryaev}.
\end{proof}

Let $U$ be another separable Hilbert space and let $L_2(U,H)$ denote the space of all Hilbert-Schmidt operators from $U$ to $H$ equipped with the Hilbert-Schmidt norm.
Let $\alpha : \mathbb{R}_+ \times \mathbb{W}(H) \rightarrow H$ and $\sigma : \mathbb{R}_+ \times \mathbb{W}(H) \rightarrow L_2(U,H)$ be mappings.

\begin{assumption}\label{ass-1}
We suppose that the following conditions are satisfied:
\begin{enumerate}
\item $\alpha$ is $\mathcal{B}(\mathbb{R}_+) \otimes \mathcal{B}(\mathbb{W}(H)) / \mathcal{B}(H)$-measurable such that for each $t \in \mathbb{R}_+$ the mapping $\alpha(t,\bullet)$ is $\mathcal{B}_t(\mathbb{W}(H))/\mathcal{B}(H)$-measurable.

\item $\sigma$ is $\mathcal{B}(\mathbb{R}_+) \otimes \mathcal{B}(\mathbb{W}(H)) / \mathcal{B}(L_2(U,H))$-measurable such that for each $t \in \mathbb{R}_+$ the mapping $\sigma(t,\bullet)$ is $\mathcal{B}_t(\mathbb{W}(H))/\mathcal{B}(L_2(U,H))$-measurable.
\end{enumerate}
\end{assumption}

We call a filtered probability space $\mathbb{B} = (\Omega,\mathcal{F},(\mathcal{F}_t)_{t \geq 0},\mathbb{P})$ satisfying the usual conditions a \emph{stochastic basis}. In the sequel, we shall use the abbreviation $\mathbb{B}$ for a stochastic basis $(\Omega,\mathcal{F},(\mathcal{F}_t)_{t \geq 0},\mathbb{P})$, and the abbreviation $\mathbb{B}'$ for another stochastic basis $(\Omega',\mathcal{F}',(\mathcal{F}_t')_{t \geq 0},\mathbb{P}')$.
For a sequence $(\beta_k)_{k \in \mathbb{N}}$ of independent Wiener processes we call the sequence
\begin{align*}
W = (\beta_k)_{k \in \mathbb{N}}
\end{align*}
a \emph{standard $\mathbb{R}^{\infty}$-Wiener process}.

\begin{definition}\label{def-martingal-sol}
A pair $(X,W)$, where $X$ is an adapted process with paths in $\mathbb{W}(H)$ and $W$ is a standard $\mathbb{R}^{\infty}$-Wiener process on a stochastic basis $\mathbb{B}$ is called a \emph{martingale solution} to (\ref{SPDE}), if we have $\mathbb{P}$--almost surely
\begin{align*}
\int_0^t \| \alpha(s,X) \| ds + \int_0^t \| \sigma(s,X) \|_{L_2(U,H)}^2 ds < \infty \quad \text{for all $t \geq 0$}
\end{align*}
and $\mathbb{P}$--almost surely it holds
\begin{align*}
X(t) = S_t X(0) + \int_0^t S_{t-s} \alpha(s,X)ds + \int_0^t S_{t-s} \sigma(s,X) dW(s), \quad t \geq 0.
\end{align*}
\end{definition}

\begin{remark}
In finite dimensions, a pair $(X,W)$ as in Definition~\ref{def-martingal-sol} is called a \emph{weak solution}. As in \cite[Chapter 8]{Da_Prato}, we use the term \emph{martingale solution} in order to avoid ambiguities with the concept of a \emph{weak solution} to (\ref{SPDE}), which means that for each $\zeta \in \mathcal{D}(A^*)$ we have $\mathbb{P}$--almost surely
\begin{align*}
\langle \zeta,X(t) \rangle = \langle \zeta, X(0) \rangle + \int_0^t \big( \langle A^* \xi, X(s) \rangle + \langle \zeta, \alpha(s,X) \rangle \big) ds + \int_0^t \langle \zeta, \sigma(s,X) \rangle dW(s)
\end{align*}
for all $t \geq 0$. Sometimes, the latter concept is also called an \emph{analytically weak solution}, see \cite{Prevot-Roeckner}.
\end{remark}

\begin{remark}
By the measurability conditions from Assumption~\ref{ass-1}, the processes $\alpha(\bullet,X)$ and $\sigma(\bullet,X)$ from Definition~\ref{def-martingal-sol} are adapted.
\end{remark}

\begin{remark}
The stochastic integral from Definition~\ref{def-martingal-sol} is defined as
\begin{align*}
\int_0^t S_{t-s} \sigma(s,X) dW(s) := \int_0^t S_{t-s} \sigma(s,X) \circ J^{-1} d\bar{W}(s), \quad t \geq 0,
\end{align*}
where $J : U \rightarrow \bar{U}$ is a one-to-one Hilbert Schmidt operator into another Hilbert space $\bar{U}$, and
\begin{align*}
\bar{W} := \sum_{k=1}^{\infty} \beta_k J e_k,
\end{align*}
where $(e_k)_{k \in \mathbb{N}}$ denotes an orthonormal basis of $U$,
is an $\bar{U}$-valued trace class Wiener process with covariance operator $Q = J J^*$. Further details about this topic can be found in \cite[Section 2.5]{Prevot-Roeckner}.
\end{remark}

\begin{definition}
We say that \emph{weak uniqueness} holds for (\ref{SPDE}), if for two martingale solutions $(X,W)$ and $(X',W')$ on stochastic bases $\mathbb{B}$ and $\mathbb{B}'$ with
\begin{align*}
\mathbb{P}^{X(0)} = (\mathbb{P}')^{X'(0)}
\end{align*}
as measures on $(H,\mathcal{B}(H))$, we have
\begin{align*}
\mathbb{P}^{X} = (\mathbb{P'})^{X'}
\end{align*}
as measures on $(\mathbb{W}(H),\mathcal{B}(\mathbb{W}(H)))$.
\end{definition}

\begin{definition}
We say that \emph{pathwise uniqueness} holds for (\ref{SPDE}), if for two martingale solutions $(X,W)$ and $(X',W)$ on the same stochastic basis $\mathbb{B}$ and with the same $\mathbb{R}^{\infty}$-Wiener process $W$ such that $\mathbb{P}(X(0) = X'(0)) = 1$ we have $X = X'$ up to indistinguishability.
\end{definition}

\begin{definition}
Let $\hat{\mathcal{E}}(H)$ be the set of maps $F : H \times \mathbb{W}_0(\bar{U}) \rightarrow \mathbb{W}(H)$ such that for every probability measure $\mu$ on $(H,\mathcal{B}(H))$ there exists a map
\begin{align*}
F_{\mu} : H \times \mathbb{W}_0(\bar{U}) \rightarrow \mathbb{W}(H),
\end{align*}
which is $\overline{\mathcal{B}(H) \otimes \mathcal{B}(\mathbb{W}_0(\bar{U}))}^{\mu \otimes \mathbb{P}^Q} / \mathcal{B}(\mathbb{W}(H))$-measurable, such that for $\mu$--almost all $x \in H$ we have
\begin{align*}
F(x,w) = F_{\mu}(x,w) \quad \text{for $\mathbb{P}^Q$--almost all $w \in \mathbb{W}_0(\bar{U})$.}
\end{align*}
Here $\overline{\mathcal{B}(H) \otimes \mathcal{B}(\mathbb{W}_0(\bar{U}))}^{\mu \otimes \mathbb{P}^Q}$ denotes the completion of $\mathcal{B}(H) \otimes \mathcal{B}(\mathbb{W}_0(\bar{U}))$ with respect to $\mu \otimes \mathbb{P}^Q$, and $\mathbb{P}^Q$ denotes the distribution of the $Q$-Wiener process $\bar{W}$ on $(\mathbb{W}_0(\bar{U}),\mathcal{B}(\mathbb{W}_0(\bar{U})))$. Of course, $F_{\mu}$ is $\mu \otimes \mathbb{P}^Q$--almost everywhere uniquely determined.
\end{definition}

\begin{definition}\label{def-mild-sol}
A martingale solution $(X,W)$ to (\ref{SPDE}) on a stochastic basis $\mathbb{B}$ is called a \emph{mild solution} if there exists a mapping $F \in \hat{\mathcal{E}}(H)$ such that the following conditions are satisfied:
\begin{enumerate}
\item  For all $x \in H$ and $t \in \mathbb{R}_+$ the mapping 
\begin{align*}
\mathbb{W}_0(\bar{U}) \rightarrow \mathbb{W}(H), \quad w \mapsto F(x,w)
\end{align*}
is $\overline{\mathcal{B}_t(\mathbb{W}_0(\bar{U}))}^{\mathbb{P}^Q} / \mathcal{B}_t(\mathbb{W}(H))$-measurable, where $\overline{\mathcal{B}_t(\mathbb{W}_0(\bar{U}))}^{\mathbb{P}^Q}$ denotes the completion with respect to $\mathbb{P}^Q$ in $\mathcal{B}(\mathbb{W}_0(\bar{U}))$.

\item  We have up to indistinguishability
\begin{align*}
X = F_{\mathbb{P}^{X(0)}} (X(0),\bar{W}).
\end{align*}
\end{enumerate}
\end{definition}

\begin{definition}\label{def-unique-sol}
We say that the SPDE (\ref{SPDE}) has a \emph{unique mild solution} if there exists a mapping $F \in \hat{\mathcal{E}}(H)$ such that:
\begin{enumerate}
\item  For all $x \in H$ and $t \in \mathbb{R}_+$ the mapping 
\begin{align*}
\mathbb{W}_0(\bar{U}) \rightarrow \mathbb{W}(H), \quad w \mapsto F(x,w) 
\end{align*}
is $\overline{\mathcal{B}_t(\mathbb{W}_0(\bar{U}))}^{\mathbb{P}^Q} / \mathcal{B}_t(\mathbb{W}(H))$-measurable, where $\overline{\mathcal{B}_t(\mathbb{W}_0(\bar{U}))}^{\mathbb{P}^Q}$ denotes the completion with respect to $\mathbb{P}^Q$ in $\mathcal{B}(\mathbb{W}_0(\bar{U}))$.

\item For every standard $\mathbb{R}^{\infty}$-Wiener process $W$ on a stochastic basis $\mathbb{B}$ and any $\mathcal{F}_0$-measurable random variable $\xi : \Omega \rightarrow H$ the pair $(X,W)$, where $X := F(\xi,\bar{W})$, is a martingale solution to (\ref{SPDE}) with $\mathbb{P}(X(0) = \xi) = 1$.

\item For any martingale solution $(X,W)$ to (\ref{SPDE}) we have up to indistinguishability
\begin{align*}
X = F_{\mathbb{P}^{X(0)}}(X(0),\bar{W}).
\end{align*}
\end{enumerate}
\end{definition}

\begin{remark}
For $A = 0$ the SPDE (\ref{SPDE}) becomes a SDE, and in this case we speak about a \emph{strong solution (unique strong solution)}, if the conditions from Definition~\ref{def-mild-sol} (Definition~\ref{def-unique-sol}) are fulfilled.
\end{remark}

\section{Proof of Theorem \ref{thm-main}}\label{sec-proof}

In this section, we shall provide the proof of Theorem~\ref{thm-main}. The general framework is that of Section~\ref{sec-framework}. In particular, we suppose that the coefficients $\alpha$ and $\sigma$ satisfy Assumption~\ref{ass-1}. As mentioned in Section~\ref{sec-intro}, we shall utilize the ``method of the moving frame'' from \cite{SPDE}. For this, we require the following assumption on the semigroup $(S_t)_{t \geq 0}$.

\begin{assumption}\label{ass-2}
We suppose that there exist another separable Hilbert space $\mathcal{H}$, a
$C_0$-group $(U_t)_{t \in \mathbb{R}}$ on $\mathcal{H}$ and
continuous linear operators $\ell \in L(H,\mathcal{H})$, $\pi \in
L(\mathcal{H},H)$ such $\ell$ is injective, we have ${\rm rg}(\pi) = H$ and ${\rm ker}(\pi) = {\rm rg} (\ell)^{\perp}$, and the diagram
\[ \begin{CD}
\mathcal{H} @>U_t>> \mathcal{H}\\
@AA\ell A @VV\pi V\\
H @>S_t>> H
\end{CD} \]
commutes for every $t \in \mathbb{R}_+$, that is
\begin{align}\label{diagram-commutes}
\pi U_t \ell = S_t \quad \text{for all $t \in \mathbb{R}_+$.}
\end{align}
\end{assumption}

\begin{remark}
According to \cite[Prop. 8.7]{SPDE}, this assumption is satisfied if the semigroup $(S_t)_{t \geq 0}$ is \emph{pseudo-contractive} (one also uses the notion \emph{quasi-contractive}), that is, there is a constant $\omega \in \mathbb{R}$ such that
\begin{align*}
\| S_t \| \leq e^{\omega t} \quad \text{for all $t \geq 0$.}
\end{align*}
This result relies on the Sz\H{o}kefalvi-Nagy theorem on unitary dilations (see e.g. \cite[Thm. I.8.1]{Nagy}, or \cite[Sec. 7.2]{Davies}). In the spirit of \cite{Nagy}, the group $(U_t)_{t \in \mathbb{R}}$ is called a \emph{dilation} of the semigroup $(S_t)_{t \geq 0}$.
\end{remark}

\begin{remark}
The Sz\H{o}kefalvi-Nagy theorem was also utilized in \cite{Seidler, Seidler2} in order to establish results concerning stochastic convolution integrals.
\end{remark}

In the sequel, for some closed subspace $K \subset H$ we denote by $\Pi_K$ the orthogonal projection on $K$.

\begin{lemma}\label{lemma-orth}
The following statements are true:
\begin{enumerate}
\item We have $\pi \ell = {\rm Id}|_H$.

\item We have $\ell \pi = \Pi_{{\rm rg}(\ell)}$ and $\ell \pi|_{{\rm rg}(\ell)} = {\rm Id}|_{{\rm rg}(\ell)}$.
\end{enumerate}
\end{lemma}

\begin{proof}
The first statement follows from (\ref{diagram-commutes}) with $t = 0$. For the second statement, note that ${\rm rg}(\ell)$ is closed, because $\ell$ is injective. Moreover, by Assumption~\ref{ass-2} we have ${\rm rg}(\ell \pi) = {\rm rg}(\ell)$ and ${\rm ker}(\ell \pi) = {\rm ker}(\pi) = {\rm rg}(\ell)^{\perp}$, showing that $\ell \pi$ is the orthogonal projection on the closed subspace ${\rm rg}(\ell)$. Consequently, we also have $\ell \pi|_{{\rm rg}(\ell)} = {\rm Id}|_{{\rm rg}(\ell)}$.
\end{proof}

Now, we introduce several mappings, namely
\begin{equation}\label{def-mappings}
\begin{aligned}
&\Gamma : \mathbb{W}(\mathcal{H}) \rightarrow \mathbb{W}(H), \quad \Gamma(w) := \pi U (w - \Pi_{{\rm rg}(\ell)^{\perp}} w(0)),
\\ &a : \mathbb{R}_+ \times \mathbb{W}(H) \rightarrow \mathcal{H}, \quad a(t,w) := U_{-t} \ell \alpha(t,w),
\\ &b : \mathbb{R}_+ \times \mathbb{W}(H) \rightarrow L_2(U,\mathcal{H}), \quad b(t,w) := U_{-t} \ell \sigma(t,w),
\\ &\bar{\alpha} : \mathbb{R}_+ \times \mathbb{W}(\mathcal{H}) \rightarrow \mathcal{H}, \quad \bar{\alpha}(t,w) := a(t,\Gamma(w)),
\\ &\bar{\sigma} : \mathbb{R}_+ \times \mathbb{W}(\mathcal{H}) \rightarrow \mathcal{H}, \quad \bar{\sigma}(t,w) := b(t,\Gamma(w)).
\end{aligned}
\end{equation}

\begin{lemma}\label{lemma-Gamma-meas}
The following statements are true:
\begin{enumerate}
\item The mapping $\Gamma$ is $\mathcal{B}(\mathbb{W}(\mathcal{H})) / \mathcal{B}(\mathbb{W}(H))$-measurable.

\item The mapping $\Gamma$ is $\mathcal{B}_t(\mathbb{W}(\mathcal{H})) / \mathcal{B}_t(\mathbb{W}(H))$-measurable for each $t \in \mathbb{R}_+$.
\end{enumerate}
\end{lemma}

\begin{proof}
Let $C \in \mathcal{C}(H)$ be a cylinder set of the form
\begin{align*}
C = \{ w \in \mathbb{W}(H) : w(t_1) \in B_1, \ldots, w(t_n) \in B_n \}
\end{align*}
with $t_1,\ldots,t_n \in \mathbb{R}_+$ and $B_1,\ldots,B_n \in \mathcal{B}(H)$ for some $n \in \mathbb{N}$. Then we have
\begin{align*}
\Gamma^{-1}(C) = \bigcap_{k=1}^n \{ w \in \mathbb{W}(\mathcal{H}) : w(t_k) - \Pi_{{\rm rg}(\ell)^{\perp}} w(0) \in (\pi U_t)^{-1}(B_k) \} \in \mathcal{C}'(H).
\end{align*}
By Lemma~\ref{lemma-cylinders}, the mapping $\Gamma$ is $\mathcal{B}_t(\mathbb{W}(\mathcal{H})) / \mathcal{B}_t(\mathbb{W}(H))$-measurable, showing the first statement. The second statement is proven analogously.
\end{proof}

\begin{lemma}\label{lemma-meas}
The following statements are true:
\begin{enumerate}
\item $\bar{\alpha}$ is $\mathcal{B}(\mathbb{R}_+) \otimes \mathcal{B}(\mathbb{W}(\mathcal{H})) / \mathcal{B}(\mathcal{H})$-measurable and for each $t \in \mathbb{R}_+$ the mapping $\bar{\alpha}(t,\bullet)$ is $\mathcal{B}_t(\mathbb{W}(\mathcal{H}))/\mathcal{B}(\mathcal{H})$-measurable.

\item $\bar{\sigma}$ is $\mathcal{B}(\mathbb{R}_+) \otimes \mathcal{B}(\mathbb{W}(\mathcal{H})) / \mathcal{B}(L_2(U,\mathcal{H}))$-measurable and for each $t \in \mathbb{R}_+$ the mapping $\bar{\sigma}(t,\bullet)$ is $\mathcal{B}_t(\mathbb{W}(\mathcal{H}))/\mathcal{B}(L_2(U,\mathcal{H}))$-measurable.
\end{enumerate}
\end{lemma}

\begin{proof}
Note that the mapping
\begin{align*}
\mathbb{R}_+ \times H \rightarrow \mathcal{H}, \quad (t,h) \mapsto U_{-t} \ell h
\end{align*}
is continuous, and hence $\mathcal{B}(\mathbb{R}_+) \otimes \mathcal{B}(H) / \mathcal{B}(\mathcal{H})$-measurable. Therefore, the claimed measurability properties of $\bar{\alpha}$ and $\bar{\sigma}$ follow from Lemma~\ref{lemma-Gamma-meas} and Assumption~\ref{ass-1}.
\end{proof}

By virtue of Lemma~\ref{lemma-meas}, we may apply the Yamada-Watanabe Theorem from \cite{Roeckner}, and obtain:

\begin{theorem}\label{thm-main-SDE}
The SDE (\ref{SDE}) has a unique strong solution if and only if both of the following two conditions are satisfied:
\begin{enumerate}
\item For each probability measure $\nu$ on $(\mathcal{H},\mathcal{B}(\mathcal{H}))$ there exists a martingale solution $(Y,W)$ to (\ref{SDE}) such that $\nu$ is the distribution of $Y(0)$.

\item Pathwise uniqueness for (\ref{SDE}) holds.
\end{enumerate}
\end{theorem}

Now, our idea for the proof of Theorem~\ref{thm-main} is as follows: The proof that the existence of a unique mild solution to the SPDE (\ref{SPDE}) implies the two conditions from Theorem~\ref{thm-main} is straightforward and can be provided as in \cite{Roeckner}. For the proof of the converse implication, we will first show that the conditions from Theorem~\ref{thm-main} imply the conditions from Theorem~\ref{thm-main-SDE}, see Propositions ~\ref{prop-proof-2} and \ref{prop-proof-3}. Then, we will apply Theorem~\ref{thm-main-SDE}, which gives us the existence of a unique strong solution to the SDE (\ref{SDE}), and finally, we will prove that this implies the existence of a unique mild solution to the SPDE (\ref{SPDE}), see Proposition~\ref{prop-proof-1}.
For the following four results (Lemma~\ref{lemma-X-Y} to Corollary~\ref{cor-Y-X}), we fix a stochastic basis $\mathbb{B} = (\Omega,\mathcal{F},(\mathcal{F}_t)_{t \geq 0},\mathbb{P})$.

\begin{lemma}\label{lemma-X-Y}
Let $\eta : \Omega \rightarrow \mathcal{H}$ be a $\mathcal{F}_0$-measurable random variable, let $(X,W)$ be a martingale solution to (\ref{SPDE}) with $X(0) = \pi \eta$, and set
\begin{align*}
Y := \eta + \int_0^{\bullet} a(s,X) ds + \int_0^{\bullet} b(s,X) dW(s).
\end{align*}
Then $(Y,W)$ is a martingale solution to (\ref{SDE}) with $Y(0) = \eta$, and we have $X = \Gamma(Y)$ up to indistinguishability.
\end{lemma}

\begin{proof}
By the definition of $Y$ we have $Y(0) = \eta$.
Moreover, since $(X,W)$ is a martingale solution to (\ref{SPDE}) with $X(0) = \pi \eta$, by identity (\ref{diagram-commutes}), Lemma~\ref{lemma-orth} and definitions (\ref{def-mappings}) we obtain $\mathbb{P}$--almost surely
\begin{align*}
X(t) &= S_t \pi \eta + \int_0^t S_{t-s} \alpha(s,X) ds + \int_0^t S_{t-s} \sigma(s,X) dW(s)
\\ &= \pi U_t \bigg( \ell \pi \eta + \int_0^t U_{-s} \ell \alpha(s,X) ds + \int_0^t U_{-s} \ell \sigma(s,X) dW(s) \bigg)
\\ &= \pi U_t \bigg( \Pi_{{\rm rg}(\ell)} \eta + \int_0^t a(s,X) ds + \int_0^t b(s,X) dW(s) \bigg)
\\ &= \pi U_t \bigg( \eta + \int_0^t a(s,X) ds + \int_0^t b(s,X) dW(s) - \Pi_{{\rm rg}(\ell)^{\perp}} \eta \bigg)
\\ &= \pi U_t (Y(t) - \Pi_{{\rm rg}(\ell)^{\perp}} Y(0)) = \Gamma(Y)(t) \quad \text{for all $t \in \mathbb{R}_+$,}
\end{align*}
showing that $X = \Gamma(Y)$ up to indistinguishability, and therefore, by (\ref{def-mappings}) we obtain up to indistinguishability
\begin{align*}
Y &= \eta + \int_0^{\bullet} a(s,X) ds + \int_0^{\bullet} b(s,X) dW(s) 
\\ &= \eta + \int_0^{\bullet} a(s,\Gamma(Y)) ds + \int_0^{\bullet} b(s,\Gamma(Y)) dW(s)
\\ &= \eta + \int_0^{\bullet} \bar{\alpha}(s,Y) ds + \int_0^{\bullet} \bar{\sigma}(s,Y) dW(s),
\end{align*}
proving that $(Y,W)$ is a martingale solution to (\ref{SDE}) with $Y(0) = \eta$.
\end{proof}

\begin{corollary}\label{cor-X-Y}
Let $\xi : \Omega \rightarrow H$ be a $\mathcal{F}_0$-measurable random variable, let $(X,W)$ be a martingale solution to (\ref{SPDE}) with $X(0) = \xi$, and set
\begin{align*}
Y := \ell \xi + \int_0^{\bullet} a(s,X) ds + \int_0^{\bullet} b(s,X) dW(s).
\end{align*}
Then $(Y,W)$ is a martingale solution to (\ref{SDE}) with $Y(0) = \ell \xi$, and we have $X = \Gamma(Y)$ up to indistinguishability.
\end{corollary}

\begin{proof}
Setting $\eta := \ell \xi$, this follows from Lemmas~\ref{lemma-orth} and \ref{lemma-X-Y}.
\end{proof}

\begin{lemma}\label{lemma-Y-X}
Let $\eta : \Omega \rightarrow \mathcal{H}$ be a $\mathcal{F}_0$-measurable random variable, let $(Y,W)$ be a martingale solution to (\ref{SDE}) with $Y(0) = \eta$, and set $X := \Gamma(Y)$. Then $(X,W)$ is a martingale solution to (\ref{SPDE}) with $X(0) = \pi \eta$, and we have up to indistinguishability
\begin{align*}
Y = \eta + \int_0^{\bullet} a(s,X) ds + \int_0^{\bullet} b(s,X) dW(s).
\end{align*}
\end{lemma}

\begin{proof}
Since $(Y,W)$ is a martingale solution to (\ref{SDE}) with $Y(0) = \eta$, by definitions (\ref{def-mappings}), Lemma~\ref{lemma-orth} and identity (\ref{diagram-commutes}) we obtain $\mathbb{P}$--almost surely
\begin{align*}
X(t) &= \Gamma(Y)(t) = \pi U_t (Y(t) - \Pi_{{\rm rg}(\ell)^{\perp}} Y(0))
\\ &= \pi U_t \bigg( \eta + \int_0^{\bullet} \bar{\alpha}(s,Y) ds + \int_0^{\bullet} \bar{\sigma}(s,Y)dW(s) - \Pi_{{\rm rg}(\ell)^{\perp}} \eta \bigg)
\\ &= \pi U_t \bigg( \Pi_{{\rm rg}(\ell)} \eta + \int_0^{\bullet} a(s,\Gamma(Y)) ds + \int_0^{\bullet} b(s,\Gamma(Y))dW(s) \bigg)
\\ &= \pi U_t \bigg( \ell \pi \eta + \int_0^{\bullet} U_{-s} \ell \alpha(s,X) ds + \int_0^{\bullet} U_{-s} \ell \sigma(s,X) dW(s) \bigg)
\\ &= S_t \pi \eta + \int_0^t S_{t-s} \alpha(s,X) ds + \int_0^t S_{t-s} \sigma(s,X) dW(s) \quad \text{for all $t \in \mathbb{R}_+$,}
\end{align*}
Therefore, $(X,W)$ is a martingale solution to (\ref{SPDE}) with $X(0) = \pi \eta$. Moreover, by definitions (\ref{def-mappings}) we get up to indistinguishability
\begin{align*}
Y &= \eta + \int_0^{\bullet} \bar{\alpha}(s,Y) ds + \int_0^{\bullet} \bar{\sigma}(s,Y) dW(s) 
\\ &= \eta + \int_0^{\bullet} a(s,\Gamma(Y)) ds + \int_0^{\bullet} b(s,\Gamma(Y)) dW(s)
\\ &= \eta + \int_0^{\bullet} a(s,X) ds + \int_0^{\bullet} b(s,X) dW(s),
\end{align*}
finishing the proof.
\end{proof}

\begin{corollary}\label{cor-Y-X}
Let $\xi : \Omega \rightarrow H$ be a $\mathcal{F}_0$-measurable random variable, let $(Y,W)$ be a martingale solution to (\ref{SDE}) with $Y(0) = \ell \xi$, and set $X := \Gamma(Y)$. Then $(X,W)$ is a martingale solution to (\ref{SPDE}) with $X(0) = \xi$, and we have up to indistinguishability
\begin{align*}
Y = \ell \xi + \int_0^{\bullet} a(s,X) ds + \int_0^{\bullet} b(s,X) dW(s).
\end{align*}
\end{corollary}

\begin{proof}
Setting $\eta := \ell \xi$, this follows from Lemmas~\ref{lemma-orth} and \ref{lemma-Y-X}.
\end{proof}

The following auxiliary result provides us with a standard extension which we require for the proof of Proposition~\ref{prop-proof-2}.

\begin{lemma}\label{lemma-std-extension}
Let $(X',W')$ be a martingale solution to (\ref{SPDE}) on a stochastic basis $\mathbb{B}'$ and let $\nu$ be a probability measure on $(\mathcal{H},\mathcal{B}(\mathcal{H}))$. Then, there exist a stochastic basis $\mathbb{B}$, a martingale solution $(X,W)$ to (\ref{SPDE}) on $\mathbb{B}$ such that the distributions of $X(0)$ and $X'(0)$ coincide, and a $\mathcal{F}_0$-measurable random variable $\eta : \Omega \rightarrow \mathcal{H}$ such that $\nu$ is the distribution of $\eta$.
\end{lemma}

\begin{proof}
We define the stochastic basis $\mathbb{B}$ as
\begin{align*}
\Omega &:= \Omega' \times \mathcal{H},
\\ \mathcal{F} &:= \overline{\mathcal{F}' \otimes \mathcal{B}(\mathcal{H})}^{\mathbb{P}' \otimes \nu},
\\ \mathcal{F}_t &:= \bigcap_{\epsilon > 0} \sigma(\mathcal{F}_{t + \epsilon}' \otimes \mathcal{B}(\mathcal{H}), \mathcal{N}), \quad t \geq 0,
\\ \mathbb{P} &:= \mathbb{P}' \otimes \nu,
\end{align*}
where $\mathcal{N}$ denotes all $\mathbb{P}' \otimes \nu$--nullsets in $\mathcal{F}' \otimes \mathcal{B}(\mathcal{H})$. Then the random variable 
\begin{align*}
\nu : \Omega \rightarrow \mathcal{H}, \quad \eta(\omega',h) := h 
\end{align*}
is $\mathcal{F}_0$-measurable and has the distribution $\nu$. We define the $H$-valued processes
\begin{align*}
X(\omega',h) := X'(\omega') \quad \text{and} \quad W(\omega',h) := 
W'(\omega').
\end{align*}
Then $W$ is a standard $\mathbb{R}^{\infty}$-Wiener process, because $W'$ is a standard $\mathbb{R}^{\infty}$-Wiener process. The independence of the increments with respect to the new filtration $(\mathcal{F}_t)_{t \geq 0}$ is shown as in the proof of \cite[Prop. 2.1.13]{Prevot-Roeckner}. Moreover, the distributions of $X(0)$ and $X'(0)$ coincide, and the pair $(X,W)$ is a martingale solution to (\ref{SPDE}), because $(X',W')$ is a martingale solution to (\ref{SPDE}).
\end{proof}

\begin{proposition}\label{prop-proof-2}
Suppose for each probability measure $\mu$ on $(H,\mathcal{B}(H))$ there exists a martingale solution $(X,W)$ to (\ref{SPDE}) such that $\mu$ is the distribution of $X(0)$.
Then, for each probability measure $\nu$ on $(\mathcal{H},\mathcal{B}(\mathcal{H}))$ there exists a martingale solution $(Y,W)$ to (\ref{SDE}) such that $\nu$ is the distribution of $Y(0)$.
\end{proposition}

\begin{proof}
Let $\nu$ be a probability measure on $(\mathcal{H},\mathcal{B}(\mathcal{H}))$. Then the image measure $\mu := \nu^{\pi}$ is a probability measure on $(H,\mathcal{B}(H))$. By assumption, there exists a martingale solution $(X',W')$ to (\ref{SPDE}) on a stochastic basis $\mathbb{B}'$ such that $\mu$ is the distribution of $X'(0)$. According to Lemma~\ref{lemma-std-extension}, there exist a stochastic basis $\mathbb{B}$, a martingale solution $(X,W)$ on $\mathbb{B}$ such that $\mu$ is the distributions of $X(0)$, and a $\mathcal{F}_0$-measurable random variable $\eta : \Omega \rightarrow \mathcal{H}$ such that $\nu$ is the distribution of $\eta$.
We set
\begin{align*}
Y := \eta + \int_0^{\bullet} a(s,X) ds + \int_0^{\bullet} b(s,X) dW(s).
\end{align*}
By Lemma~\ref{lemma-X-Y}, the pair $(Y,W)$ is a martingale solution to (\ref{SPDE}) with $Y(0) = \eta$.
\end{proof}

\begin{proposition}\label{prop-proof-3}
If pathwise uniqueness for (\ref{SPDE}) holds, then pathwise uniqueness for (\ref{SDE}) holds, too.
\end{proposition}

\begin{proof}
Let $(Y,W)$ and $(Y',W)$ be two martingale solutions to (\ref{SDE}) on the same stochastic basis $\mathbb{B}$ such that $\mathbb{P}(Y(0) = Y'(0)) = 1$. We set $X := \Gamma(Y)$ and $Y' := \Gamma(Y')$. By Lemma~\ref{lemma-Y-X}, the pairs $(X,W)$ and $(X',W)$ are two martingale solutions to (\ref{SPDE}) with $X(0) = \pi Y(0)$ and $X'(0) = \pi Y'(0)$, and we have up to indistinguishability
\begin{align*}
Y &= Y(0) + \int_0^{\bullet} a(s,X) ds + \int_0^{\bullet} b(s,X) dW(s),
\\ Y' &= Y'(0) + \int_0^{\bullet} a(s,X') ds + \int_0^{\bullet} b(s,X') dW(s).
\end{align*}
This gives us
\begin{align*}
\mathbb{P}(X(0) = X'(0)) = \mathbb{P}(\pi Y(0) = \pi Y'(0)) = 1.
\end{align*}
Since pathwise uniqueness for (\ref{SPDE}) holds, we deduce that $X = X'$ up to indistinguishability. This implies up to indistinguishability
\begin{align*}
Y &= Y(0) + \int_0^{\bullet} a(s,X) ds + \int_0^{\bullet} b(s,X) dW(s) 
\\ &= Y'(0) + \int_0^{\bullet} a(s,X') ds + \int_0^{\bullet} b(s,X') dW(s) = Y',
\end{align*}
proving that pathwise uniqueness for (\ref{SDE}) holds.
\end{proof}

%\begin{proposition}\label{prop-proof-3-add}
%If pathwise uniqueness for (\ref{SDE}) holds, then pathwise uniqueness for (\ref{SPDE}) holds, too.
%\end{proposition}

%\begin{proof}
%Let $(X,W)$ and $(X',W)$ be two martingale solutions to (\ref{SPDE}) on the same stochastic basis $\mathbb{B}$ such that $\mathbb{P}(X(0) = X'(0)) = 1$. We set
%\begin{align*}
%Y &:= \ell X(0) + \int_0^{\bullet} a(s,X) ds + \int_0^{\bullet} b(s,X) dW(s),
%\\ Y' &:= \ell X'(0) + \int_0^{\bullet} a(s,X') ds + \int_0^{\bullet} b(s,X') dW(s).
%\end{align*}
%By Corollary~\ref{cor-X-Y}, the pairs $(Y,W)$ and $(Y',W)$ are two martingale solutions to (\ref{SDE}) with $Y(0) = \ell X(0)$ and $Y'(0) = \ell X'(0)$, and we have up to indistinguishability $X = \Gamma(Y)$ and $Y' = \Gamma(Y')$. This gives us
%\begin{align*}
%\mathbb{P}(Y(0) = Y'(0)) = \mathbb{P}(\ell X(0) = \ell X'(0)) = 1.
%\end{align*}
%Since pathwise uniqueness for (\ref{SDE}) holds, we deduce that $Y = Y'$ up to indistinguishability. This implies up to indistinguishability
%\begin{align*}
%X = \Gamma(Y) = \Gamma(Y') = X',
%\end{align*}
%proving that pathwise uniqueness for (\ref{SPDE}) holds.
%\end{proof}

The following auxiliary result is required for the proof of Proposition~\ref{prop-proof-1}.

\begin{lemma}\label{lemma-ell-complete}
Let $\nu$ be an arbitrary probability measure on $(H,\mathcal{B}(H))$. We define the image measure $\nu := \mu^{\ell}$ on $(\mathcal{H},\mathcal{B}(\mathcal{H}))$. Then the mapping 
\begin{align*}
(\ell,{\rm Id}) : H \times \mathbb{W}_0(\bar{U}) \rightarrow \mathcal{H} \times \mathbb{W}_0(\bar{U})
\end{align*}
is $\overline{\mathcal{B}(H) \otimes \mathcal{B}(\mathbb{W}_0(\bar{U}))}^{\mu \otimes \mathbb{P}^Q} / \overline{\mathcal{B}(\mathcal{H}) \otimes \mathcal{B}(\mathbb{W}_0(\bar{U}))}^{\nu \otimes \mathbb{P}^Q}$-measurable.
\end{lemma}

\begin{proof}
Let $B \cup N \in \overline{\mathcal{B}(\mathcal{H}) \otimes \mathcal{B}(\mathbb{W}_0(\bar{U}))}^{\nu \otimes \mathbb{P}^Q}$ be an arbitrary measurable set with a Borel set $B \in \mathcal{B}(\mathcal{H}) \otimes \mathcal{B}(\mathbb{W}_0(\bar{U}))$ and a $\nu \otimes \mathbb{P}^Q$--nullset $N \subset \mathcal{H} \times \mathbb{W}_0(\bar{U})$. Then we have
\begin{align*}
(\ell,{\rm Id})^{-1}(B) \in \mathcal{B}(H) \otimes \mathcal{B}(\mathbb{W}_0(\bar{U})),
\end{align*}
because $(\ell,{\rm Id})$ is $\mathcal{B}(H) \otimes \mathcal{B}(\mathbb{W}_0(\bar{U})) / \mathcal{B}(\mathcal{H}) \otimes \mathcal{B}(\mathbb{W}_0(\bar{U}))$-measurable. For arbitrary Borel sets $C \in \mathcal{B}(H)$ and $D \in \mathcal{B}(\mathbb{W}_0(\bar{U}))$ we have
\begin{align*}
&(\mu \otimes \mathbb{P}^Q)^{(\ell, {\rm Id})} (C \times D) = (\mu \otimes \mathbb{P}^Q)((\ell, {\rm Id})^{-1}(C \times D)) = (\mu \otimes \mathbb{P}^Q)(\ell^{-1}(C) \times D)
\\ &= \mu(\ell^{-1}(C)) \cdot \mathbb{P}^Q(D) = \mu^{\ell}(C) \cdot \mathbb{P}^Q(D) = \nu(C) \cdot \mathbb{P}^Q(D) = (\nu \otimes \mathbb{P}^Q)(C \times D),
\end{align*}
showing that
\begin{align*}
(\mu \otimes \mathbb{P}^Q)^{(\ell, {\rm Id})} = \nu \otimes \mathbb{P}^Q.
\end{align*}
There exists a set $N' \in \mathcal{B}(\mathcal{H}) \otimes \mathcal{B}(\mathbb{W}_0(\bar{U}))$ satisfying $N \subset N'$ and $(\nu \otimes \mathbb{P}^Q)(N') = 0$. We obtain
\begin{align*}
(\mu \otimes \mathbb{P}^Q)((\ell,{\rm Id})^{-1}(N')) = (\mu \otimes \mathbb{P}^Q)^{(\ell,{\rm Id})}(N') = (\nu \otimes \mathbb{P}^Q)(N') = 0,
\end{align*}
showing that $(\ell,{\rm Id})^{-1}(N)$ is a $\mu \otimes \mathbb{P}^Q$--nullset. Consequently, we have
\begin{align*}
(\ell,{\rm Id})^{-1}(B \cup N) = (\ell,{\rm Id})^{-1}(B) \cup (\ell,{\rm Id})^{-1}(N) \in \overline{\mathcal{B}(H) \otimes \mathcal{B}(\mathbb{W}_0(\bar{U}))}^{\mu \otimes \mathbb{P}^Q},
\end{align*}
proving that $(\ell,{\rm Id})$ is $\overline{\mathcal{B}(H) \otimes \mathcal{B}(\mathbb{W}_0(\bar{U}))}^{\mu \otimes \mathbb{P}^Q} / \overline{\mathcal{B}(\mathcal{H}) \otimes \mathcal{B}(\mathbb{W}_0(\bar{U}))}^{\nu \otimes \mathbb{P}^Q}$-measurable.
\end{proof}

\begin{proposition}\label{prop-proof-1}
If the SDE (\ref{SDE}) has a unique strong solution, then the SPDE (\ref{SPDE}) has a unique mild solution.
\end{proposition}

\begin{proof}
Suppose the SDE (\ref{SDE}) has a unique mild solution. Then, there exists a mapping $G \in \hat{\mathcal{E}}(\mathcal{H})$ such that the three conditions from Definition~\ref{def-unique-sol} are fulfilled. In detail, the following conditions are satisfied:
\begin{itemize}
\item $G : \mathcal{H} \times \mathbb{W}_0(\bar{U}) \rightarrow \mathbb{W}(\mathcal{H})$ is a mapping such that for every probability measure $\nu$ on $(\mathcal{H},\mathcal{B}(\mathcal{H}))$ there exists a map
\begin{align*}
G_{\nu} : \mathcal{H} \times \mathbb{W}_0(\bar{U}) \rightarrow \mathbb{W}(\mathcal{H}),
\end{align*}
which is $\overline{\mathcal{B}(\mathcal{H}) \otimes \mathcal{B}(\mathbb{W}_0(\bar{U}))}^{\nu \otimes \mathbb{P}^Q} / \mathcal{B}(\mathbb{W}(\mathcal{H}))$-measurable, such that for $\nu$--almost all $y \in \mathcal{H}$ we have
\begin{align}\label{id-G-nu}
G(y,w) = G_{\nu}(y,w) \quad \text{for $\mathbb{P}^Q$--almost all $w \in \mathbb{W}_0(\bar{U})$.}
\end{align}
\item  For all $y \in \mathcal{H}$ and $t \in \mathbb{R}_+$ the mapping 
\begin{align*}
\mathbb{W}_0(\bar{U}) \rightarrow \mathbb{W}(\mathcal{H}), \quad w \mapsto G(y,w) 
\end{align*}
is $\overline{\mathcal{B}_t(\mathbb{W}_0(\bar{U}))}^{\mathbb{P}^Q} / \mathcal{B}_t(\mathbb{W}(\mathcal{H}))$-measurable, where $\overline{\mathcal{B}_t(\mathbb{W}_0(\bar{U}))}^{\mathbb{P}^Q}$ denotes the completion with respect to $\mathbb{P}^Q$ in $\mathcal{B}(\mathbb{W}_0(\bar{U}))$.

\item For every standard $\mathbb{R}^{\infty}$-Wiener process $W$ on a stochastic basis $\mathbb{B}$ and any $\mathcal{F}_0$-measurable random variable $\eta : \Omega \rightarrow \mathcal{H}$ the pair $(Y,W)$, where $Y := G(\eta,\bar{W})$, is a martingale solution to (\ref{SDE}) with $\mathbb{P}(Y(0) = \eta) = 1$.

\item For any martingale solution $(Y,W)$ to (\ref{SDE}) we have up to indistinguishability
\begin{align*}
Y = G_{\mathbb{P}^{Y(0)}}(Y(0),\bar{W}).
\end{align*}
\end{itemize}
We define the mapping
\begin{align*}
F : H \times \mathbb{W}_0(\bar{U}) \rightarrow \mathbb{W}(H), \quad F(x,w) := \Gamma(G(\ell x,w)),
\end{align*}
which is $\overline{\mathcal{B}(H) \otimes \mathcal{B}(\mathbb{W}_0(\bar{U}))}^{\mu \otimes \mathbb{P}^Q} / \mathcal{B}(\mathbb{W}(H))$-measurable by virtue of Lemmas~\ref{lemma-Gamma-meas} and \ref{lemma-ell-complete}.
Let us prove that $F \in \hat{\mathcal{E}}(H)$. For this purpose, let $\mu$ be an arbitrary probability measure on $(H,\mathcal{B}(H))$. We define the image measure $\nu := \mu^{\ell}$. Then $\nu$ is a probability measure on $(\mathcal{H},\mathcal{B}(\mathcal{H}))$. Furthermore, we define the mapping
\begin{align*}
F_{\mu} : H \times \mathbb{W}_0(\bar{U}) \rightarrow \mathbb{W}(H), \quad F_{\mu}(x,w) := \Gamma(G_{\nu}(\ell x,w)).
\end{align*}
There is a $\nu$--nullset $N \subset \mathcal{H}$ such that for all $y \in N^c$ identity (\ref{id-G-nu}) is satisfied.
The set $\ell^{-1}(N) \subset H$ is a $\mu$--nullset. Indeed, there is a set $N' \in \mathcal{B}(\mathcal{H})$ satisfying $N \subset N'$ and $\nu(N') = 0$. We obtain
\begin{align*}
\mu(\ell^{-1}(N')) = \mu^{\ell}(N') = \nu(N') = 0,
\end{align*}
showing that $\ell^{-1}(N) \subset H$ is a $\mu$--nullset.
Let $x \in \ell^{-1}(N)^c = \ell^{-1}(N^c)$ be arbitrary. Then we have $\ell x \in N^c$, and hence
\begin{align*}
F(x,w) = \Gamma(G(\ell x, w)) = \Gamma(G_{\nu}(\ell x,w)) = F_{\mu}(x,w)
\end{align*}
for $\mathbb{P}^Q$--almost all $w \in \mathbb{W}_0(\bar{U})$. Consequently, we have $F \in \hat{\mathcal{E}}(H)$.

Now, we shall prove that the mapping $F$ satisfies the three conditions from Definition~\ref{def-unique-sol}. For all $x \in H$ and $t \in \mathbb{R}_+$ the mapping
\begin{align*}
\mathbb{W}_0(\bar{U}) \rightarrow \mathbb{W}(H), \quad w \mapsto F(x,w) 
\end{align*}
is $\overline{\mathcal{B}_t(\mathbb{W}_0(\bar{U}))}^{\mathbb{P}^Q} / \mathcal{B}_t(\mathbb{W}(H))$-measurable due to Lemma~\ref{lemma-Gamma-meas}.

Let $W$ be a standard $\mathbb{R}^{\infty}$-Wiener process on a stochastic basis $\mathbb{B}$, and let $\xi : \Omega \rightarrow H$ be a $\mathcal{F}_0$-measurable random variable. Then the pair $(Y,W)$, where $Y := G(\ell \xi,\bar{W})$, is a martingale solution to (\ref{SDE}) with $\mathbb{P}(Y(0) = \ell \xi) = 1$. By Corollary~\ref{cor-Y-X}, the pair $(X,W)$, where $X := F(\xi,\bar{W}) = \Gamma(Y)$, is a martingale solution to (\ref{SPDE}) with $\mathbb{P}(X(0) = \xi) = 1$.

Finally, let $(X,W)$ be a martingale solution to (\ref{SPDE}) and set
\begin{align*}
Y := \ell X(0) + \int_0^{\bullet} a(s,X)ds + \int_0^{\bullet} b(s,X) dW(s).
\end{align*}
By Corollary~\ref{cor-X-Y}, the pair $(Y,W)$ is a martingale solution to (\ref{SPDE}) with $\mathbb{P}(Y(0) = \ell X(0)) = 1$, and we have $X = \Gamma(Y)$ up to indistinguishability. Denoting by $\nu$ the distribution of $Y(0)$, we have up to indistinguishability
\begin{align*}
Y = G_{\nu}(Y(0),\bar{W}).
\end{align*}
Furthermore, denoting by $\mu$ the distribution of $X(0)$, we obtain
\begin{align*}
\nu = \mathbb{P}^{Y(0)} = \mathbb{P}^{\ell X(0)} = (\mathbb{P}^{X(0)})^{\ell} = \mu^{\ell}.
\end{align*}
We deduce that up to indistinguishability
\begin{align*}
X &= \Gamma(Y) = \Gamma(G_{\nu}(Y(0),\bar{W}))
\\ &= \Gamma(G_{\nu}(\ell X(0),\bar{W})) = F_{\mu}(X(0),\bar{W}).
\end{align*}
Consequently, the mapping $F$ fulfills the three conditions from Definition~\ref{def-unique-sol}, proving that the SPDE (\ref{SPDE}) has a unique mild solution.
\end{proof}

Now, the proof of Theorem~\ref{thm-main} is a direct consequence: If the SPDE (\ref{SPDE}) has a unique mild solution, then arguing as in \cite{Roeckner} shows that the two conditions from Theorem~\ref{thm-main} are fulfilled. Conversely, if these two conditions are satisfied, then combining Propositions~\ref{prop-proof-2}, \ref{prop-proof-3}, Theorem~\ref{thm-main-SDE} and Proposition~\ref{prop-proof-1} shows that the SPDE (\ref{SPDE}) has a unique mild solution.

\section{An example}\label{sec-example}

In this section, we shall illustrate Theorem \ref{thm-main} and consider SPDEs of the type
\begin{align}\label{SPDE-ex}
dX(t) = ( A X(t) + B(t,X(t)) + F(t,X(t)) ) dt + \sqrt{Q} dW_t,
\end{align}
which have been studied in \cite{Flandoli}, with a H\"{o}lder continuous mapping $B$. We fix a finite time horizon $T > 0$, an orthonormal basis $(e_n)_{n \in \mathbb{N}}$ of $H$ and suppose (as in \cite[Section~1.1]{Flandoli}) that the following conditions are satisfied:
\begin{itemize}
\item $A$ is selfadjoint, with compact resolvent, and there is a non-decreasing sequence $(\alpha_n)_{n \in \mathbb{N}} \subset (0,\infty)$ such that $A e_n = -\alpha_n e_n$ for all $n \in \mathbb{N}$.

\item For the mapping $B : [0,T] \times H \rightarrow H$ there exist constants $L_B, M_B > 0$ and $\alpha \in (0,1]$ such that
\begin{align*}
\| B(t,x) - B(t,y) \| &\leq L_B \| x-y \|^{\alpha} \quad \text{for all $x,y \in H$ and $t \in [0,T]$,}
\\ \| B(t,x) \| &\leq M_B \quad \text{for all $x \in H$ and $t \in [0,T]$.}
\end{align*}

\item For the mapping $F : [0,T] \times H \rightarrow H$ there exists a constant $L_F > 0$ such that
\begin{align*}
\| F(t,x) - F(t,y) \| \leq L_F \| x-y \| \quad \text{for all $x,y \in H$ and $t \in [0,T]$.}
\end{align*}

\item $Q : H \rightarrow H$ is a nonnegative, selfadjoint, bounded operator such that ${\rm Tr} \, Q < \infty$ or $\sum_{n \in \mathbb{N}} \frac{\| B_n \|_{\alpha}}{\alpha_n} < \infty$, where $B_n = \langle B,e_n \rangle$ and
\begin{align*}
\| B_n \|_{\alpha} = \sup_{\genfrac{}{}{0pt}{}{t \in [0,T]}{x \in H}} \| B_n(t,x) \| + \sup_{\genfrac{}{}{0pt}{}{t \in [0,T]}{x,y \in H \text{ with } x \neq y}} \frac{\| B_n(t,x) - B_n(t,y) \|}{\| x-y \|^{\alpha}}.
\end{align*}

\item $Q_t := \int_0^t S_s Q S_s^* ds$ is a trace class operator for each $t > 0$.

\item $S_t(H) \subset Q_t^{1/2}(H)$ for each $t > 0$.

\item We have $\int_0^T \| Q_t^{-1/2} S_t \|^{1 + \theta} dt < \infty$ for some $\theta \geq \max \{ \alpha, 1-\alpha \}$.
\end{itemize}
Furthermore, in order to ensure the existence of martingale solutions, we suppose that $S_t$ is a compact operator for each $t > 0$. Then, as indicated in \cite{Flandoli}, strong existence holds true. Indeed, by \cite[Theorem~3.14]{Atma-book} we have the existence of martingale solutions, and by \cite[Theorem~7]{Flandoli} pathwise uniqueness holds true. Hence, according to Theorem~\ref{thm-main}, the SPDE (\ref{SPDE-ex}) has a unique mild solution.

\end{document}